\documentclass[a4paper,12pt]{amsart}
\usepackage{amssymb}
\usepackage{amsbsy}

\newtheorem{theorem}{Theorem}[section]
\newtheorem{lemma}[theorem]{Lemma}

\newtheorem{corollary}[theorem]{Corollary}
\newtheorem*{theorem*}{Theorem}

\theoremstyle{definition}
\newtheorem*{defn}{Definition}

\newcommand{\gO}{{\mathfrak O}}
\newcommand{\OK}{{\gO_K}}
\newcommand{\OL}{{\gO_L}}
\renewcommand{\OE}{{\gO_E}}

\newcommand{\A}{{\mathfrak{A}}}
\newcommand{\ALK}{{\mathfrak{A}_{L/K}}}
\newcommand{\AEF}{{\mathfrak{A}_{E/F}}}

\newcommand{\scrS}{{\mathcal{S}}}

\newcommand{\bmax}{{b_\mathrm{max}}}

\newcommand{\Gal}{{\rm Gal}}

\newcommand{\Z}{{\mathbb Z}}
\newcommand{\F}{{\mathbb F}}

\newcommand{\gP}{\mathfrak{P}}
\newcommand{\Q}{{\mathbb Q}}

\newcommand{\LRA}{\Leftrightarrow}

\begin{document}

\title[Galois Module Structure]{Integral Galois Module Structure for 
 Elementary Abelian Extensions with a 
Galois Scaffold}

\author{Nigel P.~Byott and G.~Griffith Elder}

\address{School of Engineering, Computer Science and
  Mathematics, University of Exeter, Exeter EX4 4QE U.K.}
\email{N.P.Byott@ex.ac.uk}

\address{Department of Mathematics \\ University of Nebraska at Omaha\\ Omaha, NE 68182-0243 U.S.A.}  
\email{elder@unomaha.edu} 
\date{February 12, 2009}

\bibliographystyle{amsalpha}

\begin{abstract}
This paper justifies an assertion in
\cite{elder} that Galois scaffolds make the questions of Galois module
structure tractable.  Let $k$ be a perfect field of characteristic $p$
and let $K=k((T))$. For the class of characteristic $p$ elementary
abelian $p$-extensions $L/K$ with Galois scaffolds described in
\cite{elder}, we give a necessary and sufficient condition for the
valuation ring $\OL$ to be free over its associated order $\ALK$ in
$K[\Gal(L/K)]$.  Interestingly, this condition agrees with the condition
found by Y.~Miyata, concerning a class of cyclic Kummer extensions in
characteristic zero.
\end{abstract}

\maketitle

\section{Introduction} \label{n1d}

Let $k$ be a perfect field of characteristic $p>0$, and let $K=k((t))$
be a local function field over $k$ of dimension 1. For any finite
extension $L$ of $K$, we write $\OL$ for the valuation ring of $L$ and
$v_L \colon L \twoheadrightarrow \Z \cup\{\infty\}$ for the normalized
valuation on $L$.  If $L/K$ is a Galois extension with Galois group
$G=\Gal(L/K)$, we write
$$ \ALK = \left\{ \alpha \in K[G] \mid 
             \alpha \OL \subseteq \OL \right\} $$
for the associated order of $\OL$ in the group algebra $K[G]$. Then
$\ALK$ is an $\OK$-order in $K[G]$ containing $\OK[G]$, and $\OL$ is
a module over $\ALK$. It is natural then to ask whether $\OL$ is a
free over $\ALK$.

This question was investigated by Aiba \cite{A} and by de Smit and
Thomas \cite{dS-T} when $L/K$ is an extension of degree $p$ (for the
analogous results in characteristic zero, see \cite{BF,BBF}).
Ramified cyclic extensions of degree $p$ in characteristic $p$ are
special in that they possess a particular property, a Galois scaffold.
In \cite{elder}, a class of arbitrarily large fully ramified
elementary abelian $p$-extensions $L/K$, the {\em near elementary
  abelian extensions}, was introduced. These extensions are similarly
special. They too possess a Galois scaffold.

\begin{defn}
Let $K=k((t))$ as above. An elementary abelian extension
$L=K(x_0,\ldots,x_{n})$ of $K$ of degree $q=p^{n+1}$ is a {\em
  one-dimensional elementary abelian extension} of $K$ if
$x_i^p-x_i=\Omega_i^{p^{n}} \beta$ for elements $\beta\in K$ and $\Omega_0=1,
\Omega_1, \ldots \Omega_{n}\in K$ such that $v_K(\beta)=-b<0$ with
$(b,p)=1$, and $v_K(\Omega_{n}) \leq \ldots \leq v_K(\Omega_1) \leq
v_K(\Omega_0)=0$, with the further condition that whenever
$v_K(\Omega_i)=\cdots =v_K(\Omega_j)$ for $i<j$, the projections of
$\Omega_i,\ldots ,\Omega_j$ into $\Omega_i\OK/\Omega_i\gP_K$ are
linearly independent over the field with $p$ elements.

More generally, $L$ is a {\em near one-dimensional elementary abelian extension} of $K$
if 
$$ x_i^p-x_i=\Omega_i^{p^{n}}\beta + \epsilon_i 
    \mbox{ for } 0 \leq i \leq n,  $$
 where $\beta$, $\Omega_0, \ldots, \Omega_{n}$ are as above, and the
 ``error terms'' $\epsilon_i\in K$ satisfy
$$ v_K(\epsilon_i) > v_K(\Omega_i^{p^{n}}\beta) + 
    \frac{(p^{n}-1)b}{p^{n}} - (p-1) \sum_{j=1}^{n-1}
    p^jv_K(\Omega_j). $$
\end{defn}

The purpose of this paper to use the Galois scaffold for near
one-dimensional elementary abelian extensions (restated here as
Theorem \ref{scaffold}) to determine a necessary and sufficient
condition for $\OL$ to be free over $\ALK$. So that we can state out
main result (Theorem \ref{assoc-main}), we introduce additional
notation.

As observed in \cite{elder}, any near one-dimensional elementary
abelian extension $L/K$ is totally ramified, and its lower
ramification numbers are the distinct elements in the sequence
\begin{equation} \label{ramnos}
  b_{(i)}= b+ p^n \sum_{j=1}^i p^jm_j 
\end{equation}
where $m_j=v_K(\Omega_{j-1})-v_K(\Omega_j)$.
This means that the first ramification number of $L/K$ is $b$, and
that all the (lower) ramification numbers are congruent modulo
$q=p^{n+1}$ to $r(b)$, the least non-negative residue of $b$.

Given any integer $j$, let
$j_{(s)}$ denote the
base-$p$ digits of $j$:
$$ j = \sum_{s=0}^\infty j_{(s)} p^s$$ with $0\leq j_{(s)}
< p$ and $j_{(s)}=0$ for $s$ large enough. Thus
$r(b)=\sum_{s=0}^n b_{(s)} p^s$.  Following \cite{NB-M}, we define
a set $\scrS(q)$.

\begin{defn} 
Given $c\in \Z$ with $(c,p)=1$, let $h=h_c$ be the unique solution of $hc
\equiv -1 \pmod{q}$, $1 \leq h \leq q-1$. Then $\scrS(q)$ consists of
all integers $c$ with $(c,p)=1$ and $1 \leq c \leq q-1$ satisfying the
following property: For all $u$, $v \geq 1$ with $u+v<b$ there exists
$s \in \{0,\ldots,n\}$ with
$$ (hu)_{(s)}+(hv)_{(s)} < p-1. $$
\end{defn}
The main result of this paper is the following
\begin{theorem} \label{assoc-main}
Let $L/K$ be any near one-dimensional elementary abelian extension
$L/K$ of degree $q=p^{n+1}$. Then $\OL$ is free over its associated order
$\ALK$ if and only if $r(b) \in \scrS(q)$. 
\end{theorem}

The definition for $\scrS(q)$ is however difficult to digest. A
simpler condition that focuses on the congruence class $r(b)$,
containing all the (lower) ramification numbers of $L/K$ can be used
to replace the condition involving $\scrS(q)$ , but at the expense of
a weaker statement:
\begin{itemize}
\item[(i)] Let $q=p^{n+1}$ with $n \leq 1$. Then $\OL$ is free over $\ALK$
  if and only if $r(b)$ divides $q-1$. 
\item[(ii)] Let $q=p^{n+1}$ with $n \geq 2$. Then $\OL$ is free over $\ALK$
 if $r(b)$ divides $p^d-1$ for some $d \in \{1, \ldots, n+1\}$. 
\end{itemize} 
See \cite{NB-M}, and note that the converse of (ii) does not always hold
\cite[\S3]{NB-M}.
This and \cite[Lemmas 5.1, 5.2]{elder} lead to the
corollaries:
\begin{corollary} \label{AS-q}
If $k$ contains the field $\F_q$ of
$q=p^{n+1}$ elements with $n\geq 0$, $K=k((t))$, and
$L=K(y)$ where 
\begin{equation} \label{yq}
    y^q-y=\beta \in K \mbox{ with } v_K(\beta)=-b<0, \quad
    (b,p)=1,
\end{equation}
then $L/K$ is a totally ramified elementary abelian extension of
degree $q$, with unique ramification break $b$. And if $r(b)$ denotes the
least non-negative residue of $b$ modulo $q$, then  $\OL$ is free over its associated order $\ALK$ if and only if $r(b)
\in \scrS(q)$. Thus
\begin{itemize}
\item[(i)] If $n \leq 1$ then $\OL$ is free over $\ALK$ if and only if
  $r(b)$ divides $q-1$. 
\item[(ii)] If $n \geq 2$ then $\OL$ is free over $\ALK$ if $r(b)$
  divides $p^d-1$ for some $d \in \{1,\ldots, n+1\}$. 
\end{itemize}  
\end{corollary}

\begin{corollary} \label{biquad}
If $k$ has characteristic 2 and $K=k((t))$ and $L$ is any totally ramified
biquadratic extension of $K$ (i.e.~$\Gal(L/K) \cong C_2 \times
C_2)$, then $\OL$ is free over $\A$. 
\end{corollary}
Corollary \ref{biquad} should be compared with the more complicated
situation in characteristic zero \cite{Ma}.

\subsection{Miyata's result in characteristic zero}
Let $F$ be a finite extension of the $p$-adic field $\Q_p$ that
contains a primitive $q=p^{n+1}$ root of unity.  Again, for any
finite Galois extension $E/F$ with Galois group $G$, we may consider
the valuation ring $\OE$ as a module over its associated order $\AEF$.
A nice, natural class of extensions consists of those
totally ramified cyclic Kummer extensions $F(\alpha)$ of degree $q$
with
\begin{equation} \label{miyata-type}
  \alpha^q =a \in F \mbox{ with } v_K(a)=t>0, \quad (t,p)=1.
\end{equation}
These extensions have been studied in a series of papers by Miyata
\cite{M1, M2, M3}.  In particular, Miyata gave a necessary and
sufficient condition in terms of $b\equiv -t\pmod{q}$ for $\OE$ to be
free over $\AEF$. This condition can be expressed in terms of
$\scrS(q)$. See \cite{NB-M}.

\begin{theorem}[Miyata] 
Let $E/F$ be as above, satisfying {\rm (\ref{miyata-type})}, then
$\OE$ is free over its associated order $\AEF$ if and only if $r(b)
\in \scrS(q)$.
\end{theorem}

This suggests that we should regard near one-dimensional elementary
abelian extensions in characteristic $p$ as somehow analogous to
Miyata's cyclic characteristic $0$ extensions.  In particular, it
seems natural to regard the families of extensions in Corollary
\ref{AS-q} and Theorem 1.4 (both defined by a single equation) to be
analogous.  If this analogy has merit, then Theorem \ref{assoc-main}
suggests that there should be a larger family of Kummer extensions,
``deformations'' of Miyata's family, for which, in some appropriate
sense, Miyata's criterion holds.

\section{Proof of Main Theorem and its Corollaries}
Recall that $L/K$ is an near one-dimensional elementary abelian
extension of characteristic $p$ local fields.

\subsection{Galois scaffold}
The definition of Galois scaffold in \cite{elder} is clarified in
\cite{elder:sharp-crit}. There are two
ingredients: A valuation criterion for a normal basis generator and a
generating set for a particularly nice $K$-basis of the group
algebra $K[G]$.

In our setting, where
$L/K$ is a near one-dimensional elementary abelian extension of
degree $q=p^{n+1}$, the valuation criterion is
$v_L(\rho)\equiv r(b)\bmod q$, which means that if $v_L(\rho)\equiv
r(b)\bmod q$ then $L=K[G]\rho$.

The second ingredient is a generating set of $\log_p|G|=n+1$ elements
$\{\Psi_i\}$ from the augmentation ideal $(\sigma-1:\sigma\in G)$ of
$K[G]$ that satisfy a regularity condition, namely
$v_L(\Psi_i^j\rho)-v_L(\rho)=j\cdot( (v_L(\Psi_i\rho')-v_L(\rho'))$
for $0\leq j<p$, and for all $\rho,\rho'\in L$ that satisfy the
valuation criterion, $v_L(\rho),v_n(\rho')\equiv r(b)\bmod q$. And
moreover, if we define $\Psi^{(a)} = \prod_{s=0}^n \Psi_s^{a_{(s)}}$
for $a=\sum_s a_{(s)} p^s$, then $\{v_L(\Psi^{(a)}\rho):0\leq a<q\}$
is a complete set of residues modulo $q$.

The main result of \cite{elder}, restated here as Theorem
\ref{scaffold},is  that a Galois scaffold exists for $L/K$.

\begin{theorem} \label{scaffold}
Let $L/K$ be a near one-dimensional elementary abelian elementary
abelian extension of degree $q=p^{n+1}$, let $G=\Gal(L/K)$ and
let 
$\bmax$ be the largest lower ramification number of $L/K$.
Then for
$0 \leq i \leq n$ there exist elements $\Psi_i$ in the augmentation
ideal $(\sigma-1:\sigma\in G)$ of $K[G]$ such
that $\Psi_i^p=0$ and, for any $\rho \in L$ with $v_L(\rho) \equiv
\bmax\equiv r(b) \pmod{q}$ and any $0\leq a<q$, we have
$$ v_L \left( \Psi^{(a)}
 \rho \right) =v_L \left( \prod_{i=0}^n \Psi_i^{a_{(i)}}
 \cdot\rho \right) =
   v_L(\rho)+\sum_{i=0}^n a_{(i)} p^i\bmax  =v_L(\rho)+a\cdot\bmax. $$
\end{theorem}
\begin{proof}
In \cite[Theorem 1.1]{elder}, take $\Psi_i=
\alpha_{n-i}(\Theta_{(i)}-1)$.
\end{proof}

In the next two sections, we describe the associated order $\ALK$ in terms of
these $\Psi_i$, and show that $\OL$ is free over $\ALK$ if and only if
$r(b) \in \scrS(q)$. To do so, we require nothing more than the
existence of the $\Psi_i$ described in Theorem \ref{scaffold}.

\subsection{Associated order}
For the fixed prime power $q=p^{n+1}$, there is a partial
order $\preceq$ on the integers $x \geq 0$  defined as follows.  Recall the
$p$-adic expansion of an
integer: $ x= \sum_{s=0}^\infty x_{(s)} p^s$ with $x_{(s)}\in
\{0,\ldots,p-1\}$. Define
$$ x \preceq y \LRA  x_{(s)} \leq y_{(s)} \mbox{ for } 0 \leq s \leq
  n. $$ 
Write $y \succeq x$ for $x \preceq y$.  Note that $\preceq$
  does not respect addition: if $0 \leq
x, y \leq q-1$ then $x \preceq q-1-y$ is equivalent to $y \preceq
q-1-x$ (both say that no carries occur in the base-$p$ addition of $x$
and $y$) but these are not equivalent to $x+y \preceq q-1$ (which
always holds).

Recall that for $a=\sum_s a_{(s)} p^s$, we have defined
$ \Psi^{(a)} = \prod_{s=0}^n \Psi_s^{a_{(s)}}$.
Since $\Psi_s^p=0$ for all $s$ we have
\begin{equation} \label{Psi-mult}
\Psi^{(a)} \Psi^{(j)} = \begin{cases} 
      \Psi^{(a+j)} & \mbox{if }a \preceq q-1-j; \\
      0 & \mbox{otherwise}.\end{cases}
\end{equation}

Now set
$ d_a = \lfloor (1+a)\bmax/q\rfloor  $
for $0 \leq a \leq q-1$. This means that $(1+a)\bmax=d_aq+r((1+a)\bmax)$ 
with $0\leq r((1+a)\bmax)<q$.
Let $\rho_*\in L$ be any element with valuation
$v_L(\rho_*)=r(b)=r(\bmax)$. 
Recall $v_K(t)=1$, so $v_L(t)=q$.
Set $\rho=t^{d_0}\rho_*$, so $v_L(\rho)=\bmax$,
and set
$$ \rho_a = t^{-d_a} \Psi^{(a)} \cdot \rho. $$
This means that based upon
Theorem \ref{scaffold}, we have
$ v_L(\rho_a) = -q d_a + v_L(\rho) + a \bmax = -q d_a +(1+a)\bmax=r(
(1+a)\bmax)$.
Using (\ref{Psi-mult}), we also have
\begin{equation} \label{Psi-rho}
\Psi^{(j)} \cdot \rho_a = \begin{cases} 
    t^{d_{a+j}-d_a} \rho_{j+a} & \mbox{if }a  \preceq q-1-j\\
    0 & \mbox{otherwise}.\end{cases}
\end{equation}

\begin{lemma}
$\{ \rho_a \}_{0 \leq a \leq q-1}$ is an $\OK$-basis for $\OL$. Moreover
  $\{ \Psi^{(a)} \}_{0 \leq a \leq q-1}$ is a $K$-basis for the group
  algebra $K[G]$, and $\rho$ generates a normal basis for the
  extension $L/K$.  
\end{lemma}
\begin{proof}
The first assertion follows from the fact that since $p\nmid\bmax$,
$v_L(\rho_a)=r( (1+a)\bmax)$ takes all values in $\{0,\ldots,q-1\}$ as
$a$ does.  From the definition of the $\rho_a$, we then deduce that
the elements $\Psi^{(a)}\cdot \rho$ span $L$ over $K$. Comparing
dimensions, it follows that $\rho$ generates a normal basis, and that
the $\Psi^{(a)}$ form a $K$-basis for $K[G]$.
\end{proof}

\subsection{Freeness over associated order}
Define 
$$ w_j = \min\{d_{a+j}-d_a \mid 0 \leq a \leq q-1, a \preceq
q-1-j \}.$$
Then $w_0=0$ and (taking $a=0$), we have $w_j \leq d_j-d_0$ for all $j$.

\begin{theorem}  \label{free-assoc}
Let $L/K$ be any near one-dimensional elementary abelian extension
$L/K$ of degree $q=p^f$, with largest ramification number $\bmax$, and
let $\rho_*\in L$ be any element with $v_L(\rho_*)=r(\bmax)$. The
associated order $\ALK$ of $\OL$ has $\OK$-basis $\{ t^{-w_j}
\Psi^{(j)} \}_{0 \leq j \leq q-1}$. Moreover $\OL$ is a free module
over $\ALK$ if and only $w_j=d_j-d_0$ for all $j$, and in this case
$\rho_*$ is a free generator of $\OL$ over $\ALK$.
\end{theorem}
\begin{proof}
Since $\{ \Psi^{(j)}\}_{0 \leq j \leq q-1}$ is a $K$-basis of $K[G]$,
any element $\alpha$ of $K[G]$ may be written $ \alpha =
\sum_{j=0}^{q-1} c_j \Psi^{(j)}$ with $c_j \in K$.  Using
(\ref{Psi-rho}) we have
\begin{eqnarray*}
  \alpha \in \ALK & 
 \LRA & \alpha \cdot \rho_a \in \OL \mbox{ for all  } a \\
 & \LRA & \sum_{j \preceq q-1-a} c_j t^{d_{j+a}-d_a} \rho_{j+a} \in
 \OL  \mbox{ for all } a \\
  & \LRA & c_j t^{d_{j+a}-d_a} \in \OK \mbox{ if } j \preceq q-1-a
  \\
  & \LRA & v_K(c_j) \geq d_a - d_{j+a} \mbox{ if } j \preceq q-1-a
  \\
  & \LRA & -v_K(c_j) \leq w_j \mbox{ for all } j.
\end{eqnarray*}
Hence the elements $t^{-w_j} \Psi^{(j)}$ form an $\OK$-basis of
  $\ALK$.

Now suppose that $w_j=d_j-d_0$ for all $j$. As $\rho_*=\rho_0$,
the definition of $\rho_j$, preceding
(\ref{Psi-rho}), yields $t^{-w_j} \Psi^{(j)} \cdot \rho_* = \rho_j$, so
the basis elements $\{t^{-w_j} \Psi^{(j)}\}_{0\leq j\leq q-1}$ take
  $\rho_*$ to the basis elements of $\{\rho_j\}_{0 \leq j \leq q-1}$
  of $\OL$. Hence $\OL$ is a free $\ALK$-module on the generator
  $\rho_*$.

Conversely, suppose that $\OL$ is free over $\ALK$, say $\OL=\ALK
\cdot \eta$ where
$ \eta = \sum_{r=0}^{q-1} x_r \rho_r $ with $x_r \in \OK$. Then $\{
t^{-w_i} \Psi^{(i)} \cdot \eta\}_{0 \leq i \leq q-1}$ is an
$\OK$-basis for $\OL$, and using (\ref{Psi-rho}) we have
$t^{-w_i}\Psi^{(i)} \cdot \eta = \sum_{0\leq r\preceq q-1-i }
x_rt^{-w_i+d_{i+r}-d_r}\rho_{i+r} $, which is an $\OK$-linear
combination of the $\rho_j$ with $j \geq i$.  In other words, there is
an upper triangular matrix $(c_{i,j})$ with $c_{i,j}\in \OK$ such that
$t^{-w_i}\Psi^{(i)} \cdot \eta = \sum_{j=i}^{q-1} c_{i,j} \rho_{j} $.
This matrix is invertible, since $\{ \rho_j \}_{0 \leq j \leq q-1}$ is
also an $\OK$-basis for $\OL$. Thus $v_K(c_{i,i})=0$ for $0\leq i<q$,
which means that $v_K(x_0t^{-w_i+d_{i+0}-d_0})=0$, and thus $w_j=d_j-d_0$
as required.
\end{proof}
  
\begin{lemma} \label{lem} 
With the above notation, $w_j=d_j-d_0$ for all $j$ if and only if
$\bmax$ satisfies
$$ r(-\bmax)+r(-i\bmax)-r(-h\bmax)>0 $$
for all integers $h$, $i$, $j$ with $0 \leq h \leq i \leq j<q$
satisfying $i+j=q-1+h$ and ${i \choose h} \not \equiv 0 \pmod{p}$. 
\end{lemma}
\begin{proof}
The condition $w_j=d_j-d_0$ for all $j$ can be restated as
\begin{equation} \label{d-ineq}
 d_{x+y}-d_x \geq d_y - d_0 \mbox{ if } x \preceq q-1-y.
\end{equation}
As this is symmetric in $x$ and $y$, we may assume $x \geq y$.

Let $i=q-1-x$, $j=q-1-y$ and $h=q-1-x-y$. So $i+j=q-1+h$.
The first step is to prove
that $0\leq y\leq x\leq q-1$ and $x \preceq q-1-y$ if and only if $0
\leq h \leq i \leq j \leq q-1$  and ${i \choose h} \not
\equiv 0 \pmod{p}$.
Observe that ${i \choose h}
\not \equiv 0 \pmod{p}$ holds if and only if  there are no carries in the
base-$p$ addition of $h$ and $i-h=y$ (see for example
\cite[p.~24]{Rib}).

Observe that $x \preceq q-1-y$ means that $x_{(s)}+y_{(s)}\leq p-1$
for all $0\leq s\leq n$. Using the definition of $h$, this means that
$h\geq 0$ and $h_{(s)}= p-1-x_{(s)}-y_{(s)}$ for all $0\leq s\leq n$.
So $0\leq y\leq x\leq q-1$ and $x \preceq q-1-y$ means that $0 \leq h
\leq i \leq j \leq q-1$ and $h_{(s)}+y_{(s)}\leq p-1$ for all $0\leq
s\leq n$. So there are no carries occur in the base-$p$ addition of
$h$ and $y$.

On the other hand, assume that $0 \leq h \leq i \leq j \leq q-1$,
$h_{(s)}+y_{(s)}\leq p-1$ for all $0\leq s\leq n$, and for a
contradiction that there is an $s$ such that $x_{(s)}+y_{(s)}\geq p$.
We may assume that $s$ is the smallest such subscript. Thus
$x_{(r)}+y_{(r)}\leq p-1$ for all $0\leq r <s$ and
$x_{(s)}+y_{(s)}=p+c_s$ where $0\leq c_s\leq p-1$.  This means that
$h_{(r)}= p-1-x_{(r)}-y_{(r)}$ for all $0\leq r<s$, and
$h_{(s)}=p-1-c_s$. So $h_{(r)}+y_{(r)}\leq p-1$ for all $0\leq r <s$
and $h_{(s)}+y_{(s)}=p-1-c_s+y_{(s)}=2p-1-x_{(s)}\geq p$.

It therefore remains to show that the inequality $d_{x+y}-d_x\geq d_y
-d_0$ corresponds to $r(-b)+r(-ib)-r(-hb)>0$, where $b=\bmax$.  For
$0\ \leq m \leq q-1$ we have $ (m+1)b = q d_m + r( (m+1)b )$,
so that
$ q(d_{x+y}-d_x) = yb + r( (x+y+1)b ) - r( (x+1)b
)$.
Hence
\begin{eqnarray*}
\lefteqn{d_{x+y}-d_x \geq d_y-d_0 } \\ 
  & \LRA & yb -r((x+y+1)b) +r((x+1)b)\geq 
    yb - r((y+1)b) + r(b) \\
  & \LRA & - r(-hb)+r(-ib) \geq -r(-jb) +q -r(-b) \\  
  & \LRA & r(-b)+r(-ib) -r(-hb)\geq  q-r(-jb). 
\end{eqnarray*}
Now $r(-b)+r(-ib) -r(-hb) \equiv  -r(-jb) \pmod{q}$ since $i+j=q-1+h$,
and $1 \leq q-r(-jb) \leq q$. Thus the last inequality is equivalent
to 
$ r(-b)+r(-ib)-r(-hb)>0$ where $b=\bmax$, 
as required.
\end{proof}
\begin{proof}[Proof of Theorem \ref{assoc-main}]
Note that, by (\ref{ramnos}), all the (lower) ramification numbers are
congruent modulo $q$. Therefore $r(-b)=r(-\bmax)$ and for all integers
$s$, $r(s\cdot r(-b))=r(-sb)$. 
As a result of Lemma \ref{lem}, we conclude that $\OL$ is free over its associated order
$\ALK$ if and only if 
\begin{equation}\label{miy-cond}
 r(-b)+r(-ib)-r(-hb)>0 
\end{equation} for all integers $h$, $i$, $j$ with $0
\leq h \leq i \leq j<q$ satisfying $i+j=q-1+h$ and ${i \choose h} \not
\equiv 0 \pmod{p}$.  This is Miyata's necessary and sufficient
condition for $\OE$ to be free over $\AEF$ (recall that $E/F$ is a
cyclic extension in characteristic $0$). 
So our conclusion agrees with Miyata's result, as recorded in \cite[Theorem
  1.3]{NB-M} where $r(-b)=t_0$, except that in this paper we have
$q=p^{n+1}$ (instead of $q=p^n$).

Note that the purpose of \cite{NB-M} was to translate Miyata's
necessary and sufficient condition, namely
(\ref{miy-cond}),
into the condition $r(b)\in \scrS(q)$
given in \cite[Theorem 1.8]{NB-M}. Thus, it is the content of the
proof for \cite[Theorem 1.8]{NB-M} that now allows us to conclude, in
our situation, that $\OL$ is free over its associated order $\ALK$ if
and only if $r(b)\in \scrS(q)$.
\end{proof}

The proof
of \cite[Theorem 1.5]{NB-M} yields the fact that for $n\leq 1$, $\OL$
is free over $\ALK$ if and only if $r(b)$ divides $q-1$.  The proof of
\cite[Theorem 1.6]{NB-M} yields the fact that for $q=p^{n+1}$ and $n\geq 2$, $\OL$ is
free over $\ALK$ if $r(b)$ divides $p^d-1$ for some $d \in \{1,
\ldots, n+1\}$.

\bibliography{bib}
\end{document}